 \documentclass[11pt]{amsart}

\usepackage{amsfonts,amssymb,verbatim,amsmath,amsthm,latexsym,textcomp,amscd}
\usepackage{graphicx}
\usepackage{color}
\usepackage{url}
\usepackage{enumerate}

\def\co{{\colon \thinspace}}

 \def\A{{\mathcal A}}

\def\I{{\mathcal I}}
\def\P{{\mathcal P}}
\def\R{{\mathcal R}}
\def\T{{\mathcal T}}

\def\DD{{\mathbb D}}

\def\RR{{\mathbb R}}
\def\ZZ{{\mathbb Z}}

\def\dd{ {\partial}}
\def\lev{\mathop{\rm  lev}}

\newtheorem{thm}{Theorem}[section]
\newtheorem{lemma}[thm]{Lemma}

\newtheorem{prop}[thm]{Proposition}
\newtheorem{definition}{Definition}

\newtheorem{remark}[thm]{Remark}

\newtheorem{introthm}{Theorem}

\title{A Pointwise Ergodic Theorem for Fuchsian Groups}

\author{Alexander I. Bufetov}
\address{\begin{flushleft} \rm Alexander I. Bufetov\\ 
\rm Department of Mathematics,
Rice University,\\ MS 136, 6100 Main Street,
Houston, Texas 77251-1892, USA
;  and \\ 
The Steklov Institute of
  Mathematics, Russian Academy of Sciences, \\Gubkina str. 8,
  119991, Moscow, Russia\\
 {\texttt{aib1@rice.edu, bufetov@mi.ras.ru \\ http://math.rice.edu/$\sim$aib1/} }\\
  \end{flushleft}}

\author{Caroline Series}
\address{\begin{flushleft} \rm Caroline Series\\ 
{\rm Mathematics Institute, 
 University of Warwick \\
Coventry CV4 7AL, UK
 \\{\texttt{C.M.Series@warwick.ac.uk \\http://www.maths.warwick.ac.uk/$\sim$masbb/} }} \end{flushleft}} 
  
\date{\today}

\begin{document}

 \begin{abstract}
 We use Series' Markovian coding for words in Fuchsian groups and the Bowen-Series coding of limit sets to prove an ergodic theorem for Ces{\`a}ro averages of spherical averages in a Fuchsian group.  \\ 
 {\bf MSC classification: 20H10, 22D40, 37A30}   
  \\
 {\bf Keywords: Ergodic theorem, Fuchsian group} 
\end{abstract}

\maketitle

\section{Introduction}
\subsection{An ergodic theorem for surface groups}
Let $g\geq 2$ and let $\Gamma$ be the fundamental group of a surface of genus $g$ endowed
with a set of generators
\begin{equation}
\label{gener}
\{a_1, \dots, a_{2g}, b_1, \dots, b_{2g}\}
\end{equation}
satisfying the standard commutator relation
$$
\prod\limits_{i=1}^{2g} [a_i, b_i]=1.
$$

For $g\in\Gamma$, let $|g|$ stand for the length of the shortest
word in the generators (\ref{gener}) representing $g$.
Let
$$
S(n)=\{g\in\Gamma: |g|=n\}
$$
be the sphere of radius $n$ in $\Gamma$, and let
$K_n$ be the cardinality of $S(n)$.
A special case of the main result of this note is the following pointwise ergodic theorem for $\Gamma$:
\begin{introthm}
\label{main}
Let $\Gamma$ act ergodically on a probability space $(X, \nu)$ by measure-preserving
transformations, and, for $g\in\Gamma$, let $T_g$ be the corresponding transformation.
Then for any $\varphi\in L_1(X, \nu)$ we have
\begin{equation}
\frac1N\sum_{n=0}^{N-1}\frac1{K_n} \sum\limits_{g\in S(n)} \varphi\circ T_g   \to \int\limits_X fd\nu
\end{equation}
both $\nu$-almost surely and in $L_1(X, \nu)$ as $N\to\infty$.
\end{introthm}
 Our main result, Theorem~\ref{generalfuchsian}, is that a similar ergodic theorem applies much more generally to any finitely generated Fuchsian group with a suitable choice of generating set, see Section~\ref{sec:final} for a precise statement.
 For numerous examples of actions to which this theorem applies, see~\cite{gor-nev, nevo}.
Theorem~\ref{generalfuchsian} is derived from a pointwise ergodic theorem for free semigroups
from \cite{bufrokh} whose formulation we shortly recall.  
For previous literature on  pointwise ergodic theorems for actions
of various classes of non-amenable discrete groups, see for example~\cite{nevo1, ns,ns1,grigor1,  mns,  Bow, gor-nev}.

\subsection{Ergodic theorems for free semigroups}

Let $(X,\nu)$ be a probability space
and let $T_1,\dots,T_m\colon X\to X$ be
$\nu$-preserving transformations.

Denote by $W_m$ the set of all finite words in the symbols $1,\dots,m$:
$$
W_m=\{w=w_1\dots w_n,\,w_i \in \{1,\dots,m\}\}.
$$
The length of a word $w$ is denoted by $|w|$.
For each $w\in W_m$, $w=w_1\dots w_n$,
define the transformation $T_w$ by the formula
$$
T_w=T_{w_1}\dots T_{w_n}. $$

Now let $A$ be an $m\times m$-matrix with non-negative entries.
For each $w\in W_m$, $w=w_1\dots w_n$, set
$$
A(w)=A_{w_1w_2}\dots A_{w_{n-1}w_n}.
$$

Now let $\varphi$ be a measurable function on $X$ and for each $n=0,1,\dots$, consider the expression
\begin{equation}\label{ssna}
s_n^A\varphi=\frac{\sum_{|w|=n}A(w)\varphi\circ T_w}{\sum_{|w|=n}A(w)},
\end{equation}
 where we assume that the denominator does not vanish.

Furthermore,  consider Ces{\`a}ro averages of the ``sphere averages'' $s_n^A$ and set:
\begin{equation}
c_N^A(\varphi)=\frac1N\sum_{n=0}^{N-1}s_n^A(\varphi).
\end{equation}

\begin{definition}
A matrix $A$ with non-negative entries is called
\textit{irreducible} if for some $n>0$ all entries of the matrix
$A+A^2+\dots +A^n$ are positive.
\end{definition}
\begin{definition}
A matrix $A$ with non-negative entries is called
\textit{strictly irreducible} if
$A$ is irreducible and
$AA^T$ is irreducible
(here $A^T$ stands for the transpose of $A$).
\end{definition}

A measurable subset $Y \subset  X$ will be called $T_1,\dots, T_m$- invariant
if its characteristic function $\chi_Y$ satisfies
$T_1\chi_Y=\dots=T_m\chi_Y=\chi_Y$.
Denote by $\mathcal{B}$ the $\sigma$-algebra of
all $T_1,\dots, T_m$- invariant subsets of $X$. Given  $\varphi\in L_1(X, \nu)$,   denote  by
${\mathbb E}(\varphi|\mathcal{B})$ the conditional expectation
of $\varphi$ with respect to ${\mathcal B}$.

We now recall Corollary 2 in \cite{bufrokh}:
\begin{prop}\label{matfun}
Let $A$ be a strictly irreducible $m\times m$ matrix.
Let $T_1,\dots,T_m$ be measure-preserving self-maps of a
probability space $(X,\nu)$. Let $\mathcal{B}$ be the
$\sigma$-algebra of $T_1,\dots,T_m$- invariant measurable sets. Then for each
$\varphi\in L_1(X,\nu)$ we have $c_N^A(T)\varphi\to {\mathbb E}(\varphi|\mathcal{B})$
almost everywhere and in $L_1(X,\nu)$ as $N\to\infty$.
\end{prop}

Theorems \ref{main} and \ref{generalfuchsian} will be derived from Proposition \ref{matfun}
using the Markov coding for Fuchsian groups introduced in~\cite{BoS}, see also~\cite{SInf, STrieste}.

\section{Markov Maps for Fuchsian Groups}
\label{sec:markov}

Let $\Gamma$  be a finitely generated non-elementary 
Fuchsian group acting in the hyperbolic disk $\DD$.
The Markov coding, originally introduced  in~\cite{BoS} to encode limit points of $\Gamma$ as infinite words in a fixed set of generators, also gives a canonical shortest form for  words in $\Gamma$. The coding   is defined relative to a  fixed choice of
fundamental domain  $\R$ for $\Gamma$, which we
suppose  is a finite sided convex polygon with
vertices contained in $ \DD \cup \dd \DD$,
such that the interior angle at each vertex is strictly less than $\pi$.
By a \emph{side} of $\R$ we mean the closure in $\DD$ of the
geodesic arc joining a
pair of adjacent vertices.
We allow the infinite area case in which some adjacent vertices on $ \dd \DD$
are joined by an arc contained in $ \dd \DD$; we do not
count these arcs as sides of $\R$.
We assume that the sides of $\R$ are paired; that is, for each side $s$
of $\R$ there is a (unique) element $e   \in \Gamma$ such that
$e(s)$ is also a side of $\R$  and such that $\R$ and $e(\R)$ are adjacent
along $e(s)$.  (Notice that this includes the possibility that
$e(s) = s$,  in which case $e$ is elliptic of order $2$
and the side $s$ contains the   fixed point of $e$ in its interior.
The condition  that the vertex angle is strictly less than $\pi$  excludes the possibility that
the fixed point of $e$ is counted as a vertex of $\R$.)

We denote by $\dd \R$ the union of the sides of $\R$, in other words,
$\dd \R$ is the part of the  boundary of $\R$ inside the disk $\DD$.
 Each side of $\dd \R$ is assigned two labels, one interior to $\R$ and one exterior, in such a way that 
the interior and exterior labels are mutually inverse elements of $\Gamma$.
We label  the side $s \subset \dd R$ interior to $\R$   by  $e$ if $e$ carries $s$ to another side $e(s)$  of $\R$, while we label the 
same side exterior to $\R$ by $e^{-1}$, see Figure~\ref{fig:labelling}. With this convention, $\R$ and $e^{-1}(\R)$ are adjacent along the side whose \emph{interior} label is $e$, while the side  $e(s)$  has interior label $e^{-1}$. 

Let $\Gamma_0$ be the set of labels of sides of
$\R $. The labelling extends to a $\Gamma$-invariant labelling of all sides of the tessellation  $\T$ of $\DD$ by images of $\R$. The conventions have been chosen in such a way that if two regions $g \R, h\R$ are adjacent along a common side $s$, then  $h^{-1}g \in \Gamma_0$ and the label on $s$ interior to $g \R$ is $h^{-1}g$, while that on the side interior to $h\R$ is $g^{-1}h$. 
Suppose that $O$ is a fixed basepoint in $\R$ and that $\gamma$ is an oriented path in $\DD$ from $O$ to $gO$, $g \in \Gamma$, which avoids   all vertices of  $\T$,   passing through in order adjacent regions  $\R = g_0\R, g_1\R,  \ldots, g_n \R = g\R$. Then the  labels of the sides crossed by $\gamma$, read in such a way that if $\gamma$ crosses from $g_{i-1}\R $ into $g_i\R$ we read off the label $e_i = g_{i-1}^{-1} g_i$ of the common side  interior to $g_i\R$, 
are in order $e_{1},  e_2, \ldots, e_n$ so that   $g= e_1 e_2 \ldots e_n$.
This proves the well known fact that  $\Gamma_0$ generates $\Gamma$, see for example~\cite{Bea}.   
In particular, if we read off labels round a   small loop round vertex $v$ of $\R$, we obtain the \emph{vertex cycle} at $v$ with corresponding \emph{vertex relation}  $ e_1 e_2 \ldots e_n = {\rm id}$. The generating sets implicit in Theorem~\ref{generalfuchsian} are obtained in this way.

\begin{figure}[hbt] 
\centering 
\includegraphics[height=5cm]{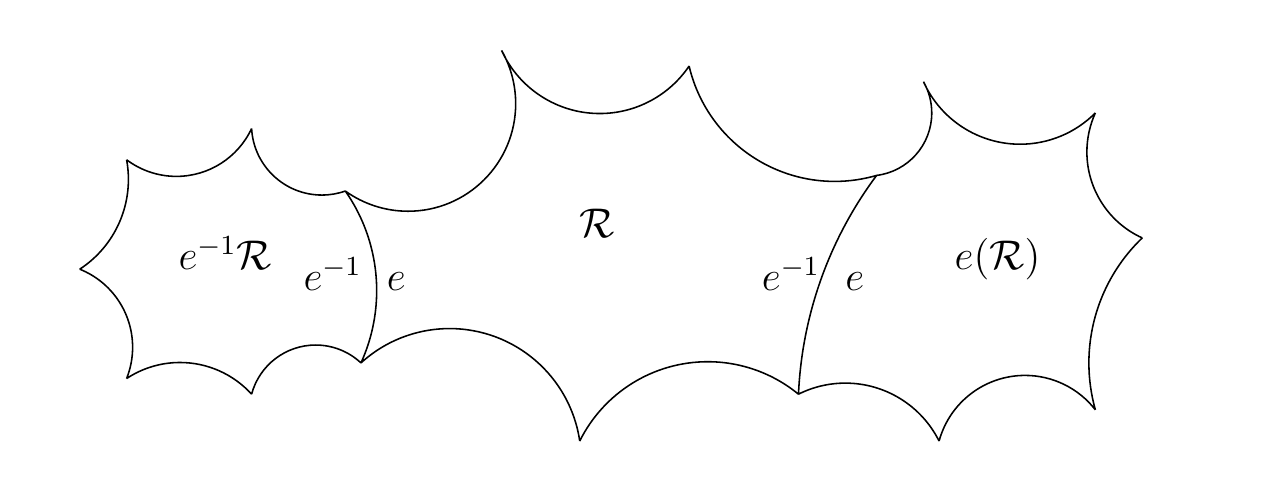} 
\caption{Labelling the sides of the fundamental domain $\R$. The label $e$ appears interior to $\R$ on  the side of $\R$ adjacent to the region $e^{-1}\R$.}
\label{fig:labelling}
\end{figure} 

Following~\cite{BoS}, the fundamental domain $\R$
is said to have \emph{even corners} if for each side $s$ of $\R$,
the complete geodesic in $\DD$ which
extends $s$ is contained in the sides of 
$\T$. 
This condition is satisfied for example, by the regular $4g$-gon of interior angle ${\pi}/{2g}$ whose sides can be paired with the generating set of Theorem~\ref{main} to form a surface of genus $g$. If the path $\gamma$ from $O$ to $gO$ described above is a hyperbolic geodesic and $\R$ has even corners, then the corresponding representation of $g$ by the word $  e_{1}  e_2 \ldots e_n$ is shortest in $(\Gamma, \Gamma_0)$, see~\cite{BiS, STrieste}.

Let $|\dd \R|$  denote the number of sides of $\R$. 
 In~\cite{BoS} we showed that, subject to certain restrictions if $|\dd \R| \le 4$,
one can associate to any  fundamental domain
with even corners
  an alphabet $\A$ and a transition matrix $P$, so  that
$\A$  is mapped by a finite-to-one map $\pi$
onto $\Gamma_0$, in such a way that the
obvious extension of $\pi$ to a map
from the set of finite sequences  with alphabet $\A$
and allowed transitions defined by $P$ to the group $\Gamma$
is surjective.  We call this map $\pi$, the  \emph{alphabet map}.
A crucial feature of the  alphabet map is that every word in its image 
  is  shortest in the word metric
on $(\Gamma,\Gamma_0)$, see~\cite{STrieste} and Theorem~\ref{thm:hypothesis} below.
In particular  $\pi$ preserves length, that is, the image under $\pi$ of a sequence
of $n$ symbols in $\A$  is an element $g \in \Gamma$ of shortest length $n$ relative to the generators $\Gamma_0$.   
Thus to sum over $g \in \Gamma$ for which $|g| = n$ as required by Theorems \ref{main} and~\ref{generalfuchsian}, we need only sum over all allowable finite words of length $n$  in the alphabet $\A$.

It follows that in order to apply Proposition~\ref{matfun}, we need  only check whether
the transition matrix $P$  is irreducible and strictly irreducible and that
$\pi$ is, in a precise sense, almost bijective to $\Gamma$.
Our main work is to show that  (subject to some  restrictions if $|\dd \R| < 5$)  this is indeed the case, see Propositions~\ref{thm:irred}, \ref{thm:strictlyirreducible}  and~\ref{prop:almostinjective}. 

Notice that the statements of Theorems~\ref{main} and ~\ref{generalfuchsian} only involve  enumerating  words in $(\Gamma,\Gamma_0)$ and are  independent of the precise geometry of $\R$. Thus for example  one can replace the regular $4g$-gon with any hyperbolic octagon whose interior angles sum to $2 \pi$ and with generators given by  the same combinatorial pattern of side pairings.    We elaborate on this observation in Section~\ref{sec:ubiquity}, where we explain why requiring a generating set which comes from a fundamental domain with even corners 
is much more general that it appears, leading to the general statement in Theorem~\ref{generalfuchsian}.

\subsection{The Coding}\label{sec:coding}

We briefly review the coding as explained in~\cite{STrieste},
in which it appears in a simpler and  more
general form than
in the original version~\cite{BoS}.

By abuse of notation  from now on we think of the tessellation $\T$ as the union of its sides, precisely  $\T= \cup \{g(\dd \R): g\in \Gamma\}$.
We assume throughout our discussion that $\R$ has even corners, so that $\T$ is a union of complete geodesics in $\DD$.
Let $\P \subset \dd \DD$ be the collection of endpoints
of those geodesics in $\T$ which meet $\dd \R$ (crucially this includes lines which meet $\dd \R$ only in a vertex of $\R$).
The points of $\P$
 partition $\partial \DD - \P$
into connected open intervals $I$; we denote  the collection of all these intervals by $\I$.

Let $s = s(e)$ be the side of $\R$ whose \emph{exterior} label is $e$.
The extension of $s$ into a complete geodesic lies in $\T$, separating $\DD  $ into two half planes, one of which contains the interior of $\R$ and one of which contains the interior of  $e\R$, see Figure~\ref{fig:labelling}. 
 Let  $L(e)$ denote the open arc on $\partial \DD$
bounding the component which contains $e\R$, see Figure~\ref{fig:sidepartition}. Each interval $I \in \I$ is contained in
$L(e)$ for at least one and  at most two sides of $\dd \R$ (see Lemma~\ref{lem:nomeeting} below for a full justification of this fact).
 For each $I \in \I$, choose $e= e(I) \in \Gamma_0$ such that
$I \subset L(e)$.
We define a map $ f\co \dd \DD - \P \to \dd \DD$ by
$f(x) = e(I)^{-1}(x)$ for $x \in I$.
Extend $f$  to a (discontinuous) possibly
 two valued map on $\P$ in the obvious way.
As observed in~\cite{BoS, STrieste}, the map $f$ is \emph{Markov} in the sense that
for any $J \in \I$, 
$f(I) \cap  J \neq \emptyset$ implies that $f(I) \supset  J $.
To see this, it is clearly
sufficient to show that
$f(\P) \subset \P$, independently of which of the possibly two  choices we make for $f$.
 So suppose $\xi \in \P$ is an endpoint of an interval $I \subset L(e)$
and   that  $f(\xi) = e(I)^{-1}(\xi)$. Write $e = e(I)$. 
 From the definitions,
  $\xi$ is  an endpoint of a geodesic $t$   which meets
the closure of the side $s = s(e)$ of $\R$.
From the definition of the labelling, $e^{-1}(s(e))= s(e^{-1})$ is also a side of $\R$.
Hence  $f(\xi) = e^{-1}(\xi)$ is an endpoint of $e^{-1}(t)$ and  $e^{-1}(t)$ must
meet the closure of $s(e^{-1})$, hence $f(\xi) \in \P$.

 \subsubsection{The alphabet map}
 We define our alphabet by setting  $\A = \I$, in other words, $\A$ is the collection of all the intervals defined by the subdivision of $\dd \DD$ by points in $\P$.
We define a
  transition matrix $P = (p_{I,J} )$
  by $p_{I,J} = 1$ if $f(I) \supset  J$ and $p_{I,J} = 0$ otherwise.
  Let
$\Sigma_F$ denote  the set of finite sequences $I_{i_0}\ldots  I_{i_n}$
with $I_{i_r} \in \I$ such that  $p_{I_{i_{r-1}},I_{i_{r}}}>0$
 for all $r= 0,\ldots, n-1$.
Thus $\Sigma_F$ consists of all allowable finite sequnces
in the subshift  on the symbols $I \in \I$ with transition rule specified by
$P$.

The \emph{alphabet map}  $\pi \co \I \to \Gamma_0$ is the map which associates to each $I \in \I$
the element $e \in \Gamma_0$ corresponding to our choice of $e$ for which $ I \subset L(e)$, equivalently for which $f_{| I} = e^{-1}$.
This extends in an obvious way to a map
$\pi: \Sigma_F \to \Gamma$.
Recall that a product of $n$ elements of  $\Gamma_0$ is \emph{shortest}
(with respect  to the generators $\Gamma_0$ ) if
it cannot be expressed as a product of less than $n$
elements of $\Gamma_0$.
An important feature of  the coding is the
following result which follows from  Theorem 5.10 and Corollary 5.11 in~\cite{STrieste},
see also Theorem 2.8 in~\cite{BiS}.

\begin{thm}
\label{thm:hypothesis} Suppose that $\R$ has even corners and that either
\begin{enumerate}
\renewcommand{\labelenumi}{(\roman{enumi})}
\item $|\dd \R| \ge 5$ or
\item  $|\dd \R| = 4$  and, if in addition all vertices of $\R$ lie
in $ \DD$, then at least three
geodesics in $\T$ meet at each vertex or
\item  $|\dd \R| = 3$  and at least one vertex of $\R$ is on $\dd \DD$.
\end{enumerate}
Then  the alphabet map  $\pi$ is surjective to $\Gamma$. Moreover every word in $\pi (\Sigma_F)$ is shortest with respect to the geometric generators $\Gamma_0$ associated to $\R$ as above, and each element $g \in \Gamma$ has a \emph{unique} representation   in  $\pi (\Sigma_F)$.
\end{thm}

In what follows, we always assume that $\R$ satisfies one of the hypotheses
of Theorem~\ref{thm:hypothesis}. Uniqueness means that any $g \in \Gamma$ has a unique representation as  a word $ e_{i_1} \ldots e_{i_n}$  in the image of $\pi$; however we may have two distinct sequences $I_{i_1} \ldots I_{i_n}$ and $I_{j_1} \ldots I_{j_n}$
with $\pi(I_{i_r} )= \pi( I_{j_r})$ for all $r$. A key point   in the proof of Theorem~\ref{generalfuchsian} is to show
 that $\pi$ is nevertheless almost injective, precisely:
\begin{prop}
\label{thm:almostinjective}  Let
$g \in \Gamma$.
 Then  $\pi^{-1}(g) \le 2$ and
$\#\{ g\co |g|=n, \pi^{-1}(g)|>1 \}/n \to 0$ as $ n \to \infty$.
\end{prop}

\subsubsection{Ubiquity of even corners} \label{sec:ubiquity}
The condition that $\R$ have even corners may seem very special. However our result
depends only on the combinatorics of the generating set, so  that regions which do not have even corners may still have side pairings which satisfy the required conditions.
More precisely, we note the following facts:
\begin{enumerate}
\item Many standard fundamental domains, for example the symmetrical $4g$-gon
for a surface of genus $g$, and the standard fundamental domain for $SL(2,\ZZ)$,
do have even corners. Moreover the condition depends only on the geometry of $\R$ and not on the pattern of side pairings, so that for example the symmetrical $4g$-gon with opposite sides paired would work equally well.
\item A simple observation going back to Koebe  shows that
the group corresponding to
any closed hyperbolic surface has a fundamental domain with even corners.
To see this we have only to choose smooth closed geodesics for the sides
of $\R$. This is always possible; see~\cite{BiS} for a picture.
\item If  $\Gamma$  has no torsion but contains parabolics
then one can always choose a fundamental polygon
with all vertices on $\dd \DD$. Such a polygon certainly has even corners (and in fact $\Gamma$ is then a free group).
\item We showed in~\cite{BoS} that every finitely generated Fuchsian group
is quasiconformally equivalent to one which has a fundamental domain with even corners.
The deformation can be chosen to preserve  the combinatorial pattern of sides and
side pairings of $\R$.
Since our   results only depend on the group and not on the specific
hyperbolic structure, it is sufficient to work with the deformed group
for which the fundamental domain does have the even corner property.
\item We show in~\cite{SInf} that, subject to hypotheses essentially
the same as those of Theorem~\ref{thm:hypothesis}, one can always find
a partition of $\dd \DD$ and a map $f$ whose combinatorial properties
are identical to those which pertain when $\R$ has even corners.
One could work directly in this setting, but with the disadvantage
that without the geometry of $\R$
at ones disposal it would be much harder to follow what is going on.
\item Despite the above comments, one should be clear that our results
depend heavily on the choice of $\R$ and the geometrical generating set
$\Gamma_0$.
\end{enumerate}

 \begin{remark}
  {\rm Let $\Sigma$ denote the space of all infinite sequences $I_{i_0} I_{i_1} \ldots$
allowed by the transition matrix $P$. Let $\Lambda(\Gamma)$  denote the limit set of
$\Gamma$.
We showed in~\cite{SInf} that, modulo the   exceptional cases
 excluded by Theorem~\ref{thm:hypothesis},
the obvious map defined by ``$f$-expansions'' induces
  a  surjection  $\pi(\Sigma) \to \Lambda(\Gamma)$
which is injective except on a countable number of points
where it is two-to-one. (The exceptional points are essentially the endpoints
of infinite special chains, see~\cite{SInf} Proposition 4.6 and
Section~\ref{sec:alphabet} below.) 

If $\Gamma$ contains no parabolics, then we show in~\cite{SInf}
that   Hausdorff measure in dimension $\delta$, where $\delta$
is the exponent of convergence of $\Gamma$, lifts to a Gibbs measure
on  $\Sigma$.
In this case, Hausdorff measure is the so-called Patterson
 measure on $\Lambda(\Gamma)$. In particular, if $\DD/ \Gamma$ is compact then
Lebesgue measure on $\dd \DD$ is   Gibbs.  
In~\cite{SMartin} we studied random walks on the Cayley graph of $\Gamma$.
We showed
that if the transition probabilities are finitely supported
on $\Gamma$, then   hitting  distribution on $\Lambda(\Gamma)$ is Gibbs. Note however that the obvious actions of $\Gamma$ on these spaces are not measure preserving, so our ergodic theorem does not apply in this context.}
 \end{remark}

\subsection{Irreducibility}

In this section we show that the
transition matrix $P$ associated to the map $f$ is irreducible, which unfortunately means delving in some detail into its mechanics.
The starting point is
the following simple but crucial result which is Lemma 2.2 in~\cite{BoS}.

 \begin{lemma}
 \label{lem:nomeeting} Suppose that
 $|\dd \R| > 3$. Then any two distinct complete geodesics in $ \T$
 which meet $\dd \R$
 are either disjoint in $  \DD \cup \dd \DD$, or intersect in a vertex of $\R$.
 \end{lemma}
The most interesting case is when lines  $t, t' \subset \T$ meet  $\dd \R$  in
the two vertices at the opposite ends  of some side $s$,  neither being coincident with $s$. The lemma asserts that
$t$ and $t'$  are disjoint. Note also that there is a choice in the definition  of $f$ only if $L(e) \cap L(d) \neq \emptyset$. It follows from Lemma~\ref{lem:nomeeting} that this occurs only if $s(e)$ and $s(d)$ are adjacent sides of $\R$.

\begin{figure}[hbt]
\centering 
\includegraphics[height=6cm]{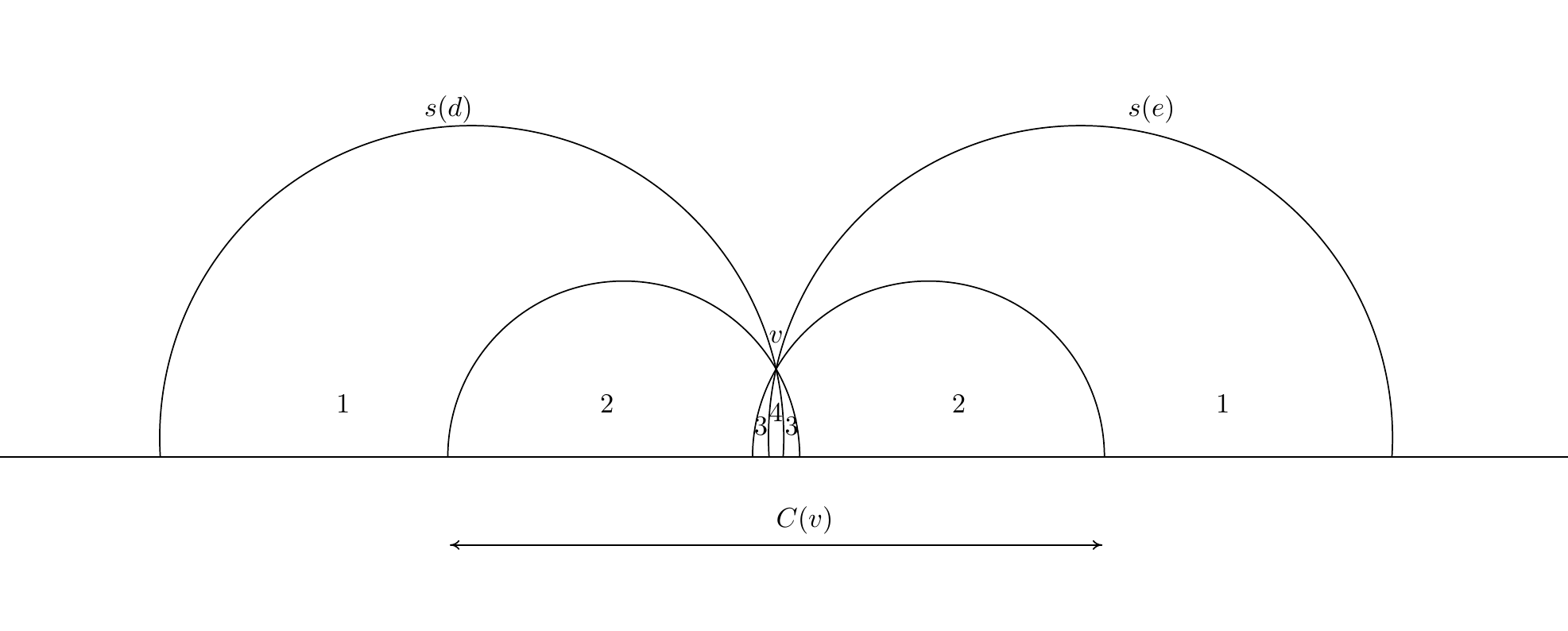} 
\caption{Intervals round a vertex $v$ of $\R$. In the case shown, $n(v) = 4$, numbers indicating the levels. All intervals except those of level $0$ and $1$ are contained in $\I$. The crown $C(v)$ is the union of all intervals of levels $2, \ldots, n(v)$.}
 \label{fig:levels}
\end{figure} 

 Now we establish some  notation.
 Let $v$ be a vertex of $\R$.
 Let $n(v)$ be the  number of geodesics of $\T$ which meet at $v$ (so that
 $2n(v)$ copies of $\R$ meet at $v$).
 The endpoints on $\dd \DD$ of the $n(v)$ complete geodesics in $\T$ which meet at $v$
 partition their complement in $\dd \DD$ into $2n(v)$ open intervals,
each of which is (the interior of the closure of) a union of intervals
$J \in \I$. We assign to each of
  these new intervals a level, denoted $\lev(.)$,  as follows.
 The interval $L(d) \cap L(e)$
  bounded by the
  endpoints of the extensions of  sides    $s(d)$ and $s(e)$ of $\R$ which meet  at $v$ has level  $n(v)$.
   The interval opposite $L(d) \cap L(e)$
   has level $0$.
  The remaining intervals going round in both directions
  from $0$ to $n(v)$ have in order levels $1,2 , \ldots, n(v)-1$, see Figure~\ref{fig:levels}.
 It follows from Lemma~\ref{lem:nomeeting} that the intervals of all levels except $0$ and
   $1$ are also intervals in the set $\I$, and moreover that 
  there is a choice in the definition  of $f$ only on intervals of level $n(v)$.
  Finally define the \emph{crown} of $v$, denoted
 $C(v) $,  to be the interior of the
closure of the union of the  intervals of levels $n(v), n(v) -1, \ldots, n(v) -2$ at $v$,
see Figure~\ref{fig:levels}.
 Thus $C(v) $ is an open interval on $\dd \DD$.
 Notice that if $v \in \dd \DD$ then $C(v) = \emptyset$.

\begin{figure}[hbt]
\centering 
\includegraphics[height=5cm]{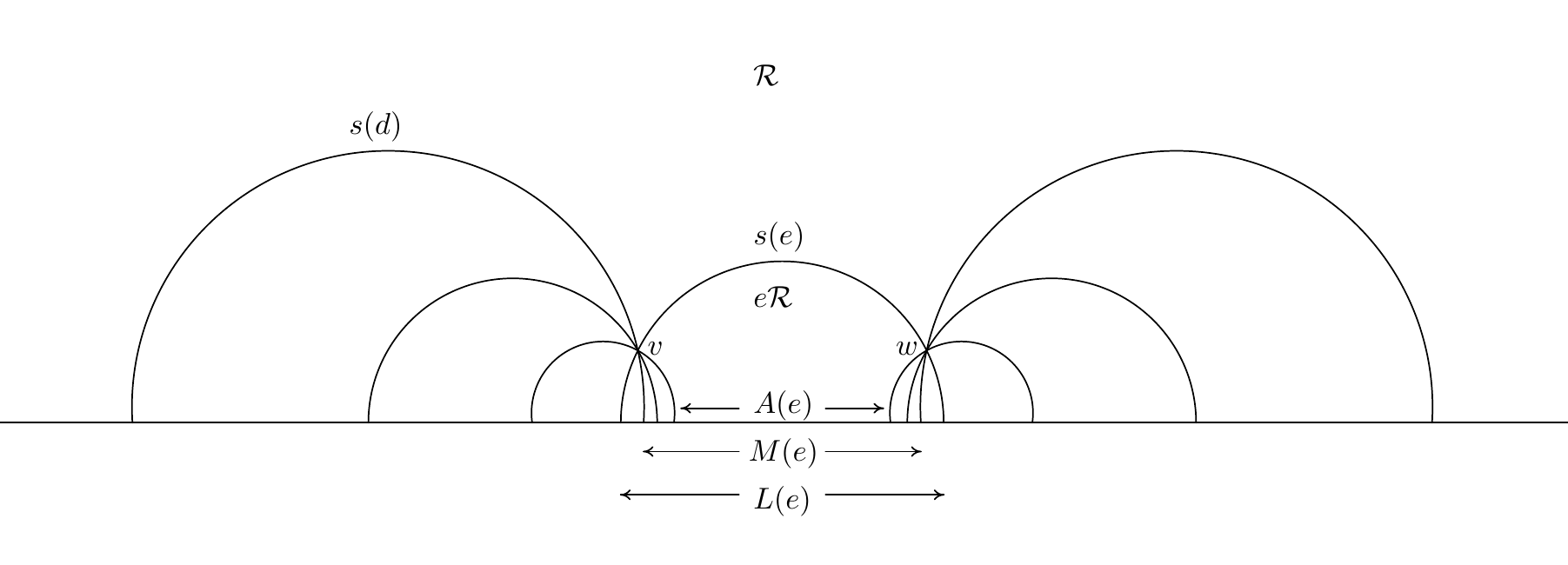} 
\caption{Intervals subtended by a side of $\R$. If the vertices of $s(e)$ are $v,w$ then $A(e)$ is the subset of $L(e)$ which is contained in neither $C(v)$ nor $C(w)$, while $M(e)$ consists of all of $L(e)$ except the `ambiguous' top level intervals at either end of $L(e)$. Thus $f_{|M(e)} = e^{-1}$ while $f_{|L(e) \cap L(d)}$ can have value either $e^{-1}$ or $d^{-1}$. }
 \label{fig:sidepartition}
\end{figure} 

 For each $e \in \Gamma_0$  define $\dd e $ to be the two vertices
of $\R$  at the ends of the side  $s(e)$, and let
 $C (e) = C(v) \cup C(w)$ where $\dd e = \{v,w\}$.
Also let  $M(e)= L(e) - \overline {\cup_{ d \neq e} L(d)} $
 and  $A(e)  = L(e) - \overline {C (\dd e)}$, see Figure~\ref{fig:sidepartition}.
If $|\dd \R| >3$ then Lemma~\ref{lem:nomeeting} implies that
 $A(e) \neq \emptyset$.
Note that if  $x \in M(e)$ then $f(x) = e^{-1}(x)$.

We are finally ready to start our proof that the transition matrix $P$ is irreducible.
In what follows we shall say that a
constant  depends only on $\R$,
when we mean that it depends on $\R$ and the
combinatorial pattern of side pairings of $\R$.
We introduce various such constants and denote all of them by $K$.

\begin{lemma}
 \label{lem:BtoA}
Suppose $| \dd \R| >3$. Then there exists $K \in \mathbb N$, depending only on $\R$,
such that for any  $I \in \I$, we have
  $f^r(I) \supset A(e)$ for some $r$ with $0 \le r \le K$
 and some $e \in \Gamma_0$.
 \end{lemma}
 \begin{proof}  If $I$ is not already of the form $A(e)$ for some $e \in \Gamma_0$,
 then it is in the crown of some vertex $v$ of $\R$, and hence  $\lev(I) = r >1$.
  Suppose that $ I \subset L(e)$ and that $f_{|I} = e^{-1}$. Then $f$ carries  $s(e)$ 
 to the side $e^{-1}(s) = s(e^{-1})$, so that $f(v) = e^{-1} v$ is a vertex
 of $s(e^{-1})$. Following the discussion at the start of  Section~\ref{sec:markov}, the cyclic order of labels around the vertices $v$ and  $f(v)$
 is the same, and 
by inspection one sees that
$f(I)$ is an interval of level $r-1$ at $e^{-1}v$.
 Since the sets $A(e)$ are contained in the union of
 level $1$ sets associated to all vertices of $\R$, the result follows.
 \end{proof}

\begin{lemma}
 \label{lem:MtoAC}
  Let $e \in \Gamma_0$. Then
\begin{enumerate}
\renewcommand{\labelenumi}{(\roman{enumi})}
\item $f(M(e)) \supset C(v)$ for any vertex $v \notin \dd e^{-1}$.
\item $f(M(e)) \supset A(d)$ for any  $ d \neq   e^{-1}$.
\end{enumerate}
 \end{lemma}
 \begin{proof}
By definition,  $f_{| M(e)} = e^{-1}$  so that $f$ carries $s(e)$ to $s(e^{-1})$. Moreover 
$e^{-1}$ carries
$L(e)$  to the complement in $\dd \DD$ of $L(e^{-1})$. We need to find the image under 
$e^{-1}$ of 
$M(e) \subset L(e)$.
 Let $V$,$W$ denote the endpoints of $L(e^{-1})$ on $\dd \DD$ and let $V_1,W_1$ denote the  points in $\P$  adjacent
to $V$,$W$ and outside $L(e^{-1})$.
Then $e^{-1}(M(e))$ is the  interval on $\dd \DD$ bounded by $V_1,W_1$
and not containing $L(e^{-1})$.
This covers all of $\dd \DD$ except for $A(e^{-1})$ and parts of $C(\dd e^{-1})$.
\end{proof}

\begin{lemma}
 \label{lem:AtoM}
 Suppose that $|\dd \R| >3$. Then there exists $K \in \mathbb N$, depending only on $\R$,
   such that  for any $e \in \Gamma_0$, and any $I \in \I$ which is contained in
 $ M(e)$, we have
    $f^r(A(e)) \supset I$
   for some $r \le K$.
 \end{lemma}
 \begin{proof}
By definition,  $f_{| A(e)} = e^{-1}$, and $e^{-1}$ maps $e\R$ (which is the copy of $\R$ adjacent to $\R$ along $s= s(e)$)  to $\R$.   The endpoints of $A(e)$ are the endpoints on
$\dd \DD$ of the extensions of the sides of  $e\R$  adjacent to $s$, see Figure~\ref{fig:sidepartition}. 
Thus the endpoints of $f(A(e))$ are the endpoints on $\dd \DD$ of the extensions of the sides of  $\R$  adjacent to $e^{-1}(s) = s(e^{-1})$. 
If these sides are $s(d),s(d')$, then provided that $e \neq e^{-1}$, we see that $f(A(e))$ is the complement in $\dd \DD$ of $L(e^{-1}) \cup L(d) \cup L(d')$. 
 This gives the result (with $r=1)$ in
the   case in which $s(e)$ is neither equal to nor   adjacent to $s(e^{-1})$. Since in both of the exceptionl cases $e$ is necessarily elliptic, this  in particular proves the result whenever $\Gamma$ is torsion free.

Now suppose that  $s(e)$ and  $s(e^{-1})$ are adjacent
with common vertex $v \in \DD$, so that $e$ is elliptic with fixed point $v$. Reasoning as above, we see that $f(A(e))$ covers $M(x)$ for any side $s(x)$  neither equal nor adjacent  to
$s(e^{-1})$. Fix one such $x$, which is possible since $|\dd \R|>3$.
Since by our assumption $s(e)$ and  $s(e^{-1})$ are adjacent, we have $x \neq e,e^{-1}$.
By Lemma~\ref{lem:MtoAC} (i), $f(M(x))$ covers all crowns except
$C(w)$ for $w \in \dd x^{-1}$. In particular,
$f(M(x)) \supset C(v)$.
Let $I \in \I$ be the level $2$ interval contained in $C(v) \cap L(e)$. Then $f(I)$ is the level $1$ interval in $C(v) \cap L(e)$,  which is equal to $L(e) - C(v)$.
  Thus suitable powers of $f$ applied to  $C(v)$
cover all of  $L(e) \supset M(e)$ which gives the result.



Now consider the case $e= e^{-1}$.
First    assume that
$\dd R$ has at least $5$ sides.
(Remember we count the edge containing the fixed point of $e$ as one side.)
In this case, there exist
$x,y$, distinct from each other and  from $e$, such that
$f(A(e))$ covers $M(x)$ and $M(y)$.
Now $f(M(x))$ covers  $A(e)$ since $e \neq x^{-1}$.
In addition,
$f(M(x))$ and $f(M(y))$ together cover  all crowns except  for those crowns
$C(w)$ with $w \in \dd x^{-1} \cap \dd y^{-1}$.
This implies that $f(M(x)) \cup f(M(y))$ covers
$C(v)$ for $v \in \dd e$, which gives the result.

Finally,  suppose that
$\dd R$ has  $4$ sides.
In this case, $f(A(e))$ only covers $M(x)$ for
$x$ the side opposite $e$.
As usual, $f(M(x))$ covers  $A(e)$.
If $x=x^{-1}$  then $f(M(x))$ covers $M(e)$. 
Otherwise, $x^{-1}$ is adjacent to $e$ and $f(M(x))$ covers
$M(y)$ where $y$ is the fourth side of $\dd R$ (opposite $x^{-1}$).
In this case we have $y = y^{-1}$. Letting $v,w$ be the vertices of $s(e)$ adjacent to $s(x^{-1})$ and $s(y)$  respectively, we see that $f(M(y))$ covers $C(v)$ and  $f(M(x))$ covers $C(w)$.  The result follows.  
\end{proof}

\begin{prop} \label{thm:irred} The Markov chain $P$  is irreducible.
\end{prop} \begin{proof} We have to show that there exists $K$, depending only on $\R$, such that
for any $I,J \in \I$, we have $f^r(I) \supset J$ for some
$0 \le r \le K$.

Assume first that
$|\dd R| >3$.
By Lemma~\ref{lem:BtoA}, we may as well
assume that $I = A(e)$ for some $e \in \Gamma_0$.
By Lemma~\ref{lem:AtoM}, it will be enough to show that images of
$M(e)$ cover $\dd \DD$.
By Lemma~\ref{lem:MtoAC}, $f(M(e))$ covers all crowns except $C(\dd e^{-1})$ and
all sets $  A(x)$ with $x \neq e^{-1}$.
Since $|\dd \R| >3$ we may choose $x,y$
distinct from each other and from $e$ and $e^{-1}$ such that
$f(M(e)) \supset A(x) \cup A(y)$.
By Lemmas~\ref{lem:AtoM} and~\ref{lem:MtoAC}, there exists $r<K$ such that
$f^r(A(x))$ covers $A(e^{-1})$
and  all crowns except $C(\dd x^{-1})$.
Likewise $f^s(A(y))$ covers  all crowns except $C(\dd y^{-1})$ for some $s<K$.
Now by choice $x,y$ and $e$ are distinct and so
$C(\dd x^{-1}) \cap C(\dd y^{-1}) \cap C(\dd e^{-1})  = \emptyset$.
The result follows.

Finally, we have to consider the case in which
$|\dd R| =3$.
Notice that this is the only case in which it is possible that
$A(e) = \emptyset$.
  By hypothesis, at least one vertex of $\R$ is on $\dd \DD$.
There are only three possible cases:
\begin{enumerate}
\renewcommand{\labelenumi}{(\roman{enumi})}
\item $\R$ has three vertices on $\dd \DD$;
\item $\R$ has two vertices $v,w$ on $\dd \DD$. The side joining $v,w$ is
paired to itself
by an order two elliptic $x$; the remaining two sides are paired to each other
by an elliptic $e$ with fixed point at the third (finite) vertex $u$;
\item $\R$ has one vertex $v$ on $\dd \DD$. The two sides meeting at $v$ are paired to each other
by $e$, the third side is paired to itself by an order two elliptic $x$.
\end{enumerate}

Case (ii). Set $B(e^{\pm}) = L(e^{\pm}) - C(u)$. Note that $f(B(e^{\pm})) = L(x)$
and $f(L(x)) = \dd \DD - L(x)$. Furthermore, for each $I \subset C(u)$ it is clear
that $f^r(I) = (B(e^{\pm}))$ for some $r \le n(u)$.
This proves the result.

Case (iii).  Let $u$ be the finite endpoint of side $e$ and let
$J(u) = C(u) - L(e)$. Define $J(w)$ similarly relative to
$w$  the finite endpoint of $e^{-1}$. Note that $f(J(u)) \supset A(e^{-1})$
and $f(J(w)) \supset A(e)$.
It follows that the image of every interval in $C(u) \cup C(v)$
eventually covers 
either $A(e)$ or $A(e^{-1})$.
Further, a  bounded image of $A(e)$  covers $L(e ) \cup J(u)$
and a  bounded image of $A(e^{-1})$  covers $L(e^{-1} ) \cup J(w)$.
The result follows.

Case (i) is easier and is left to the reader.
\end{proof}

\subsection{Strict irreducibility}

We now investigate strict irreducibility of the transition matrix $P$.
It is well known and easy to see  that $P$ is strictly irreducible
 if the equivalence relation $\sim$ on $\I$ generated
by $ I \sim J$ if   $f(I) \cap f(J) \neq \emptyset$
has just one equivalence class.
We show that if $|\dd \R| >4$ then  $f$ is always strictly irreducible,
while if  $|\dd \R| \le 4$ the map  $f$ may or may not be strictly irreducible
depending on the precise arrangement of   $\R$ and its side pairings. In
particular, the continued fraction map associated to the standard fundamental
domain for $SL(2,\ZZ)$  is \emph{not} strictly irreducible.

\begin{lemma} The Markov chain associated to any choice of
 Markov map $f$ for the fundamental domain
$|z| >1, -1/2 < \Re z < 1/2$ is not strictly irreducible.
\end{lemma}
\begin{proof} The continuations of the sides of $\dd \R$
through the two vertices at $(1 \pm \sqrt 3 i ) /2$
meet the real axis $\RR$
in the $7$ points $-2,-1,-1/2,0,1/2,1,2$
which partition $\RR$
into $8$ intervals which we number $\bf {1}, \bf {2} \ldots, \bf {8}$
in order from left to right,
thus for example $\bf{3}$ denotes the interval  $(-1,-1/2)$.
The map $f$ is defined as:
\begin{eqnarray*}
f(x)&= &x+1 \ \ {\rm for} \ \ x \in \bf {1} \cup \bf {2},\\
f(x) &= &-1/x \ \ {\rm for} \ \ x \in \bf {3}  \cup \bf {4} \cup \bf {5} \cup  \bf {6}, \\
f(x)& = &x-1 \ \ {\rm for} \ \ x \in \bf {7} \cup \bf {8}.
\end{eqnarray*}
There is a choice for  $f$ on the overlap regions
$\bf {4} $ and $\bf {5}$: for definiteness we have taken the usual choice
$f(x) =  -1/x$ for $x \in \bf {4} \cup \bf {5}$
which is associated to the continued fraction map.

It is easy to write down the transition matrix $P$ for $f$.
We find $\bf {1} \rightarrow  \bf {1} \cup \bf {2}$,
$\bf {2} \rightarrow  \bf {3} \cup \bf {4}$,
$\bf {3} \rightarrow  \bf {7}$,
$\bf {4} \rightarrow  \bf {8}$,
$\bf {5} \rightarrow  \bf {1}$,
$\bf {6} \rightarrow  \bf {2}$,
$\bf {7} \rightarrow  \bf {5} \cup \bf {6}$,
and $\bf {8} \rightarrow  \bf {7} \cup \bf {8}$.
From this we easily see that there are four equivalence classes under $\sim$:
$\{\bf {3}, \bf {4} , \bf {8} \}$,
$\{\bf {5} , \bf {6} , \bf {1}\}$,
$\{\bf {2} \}$
and $\{\bf {7}\}$.
This gives the result.
We remark that even had we made the other choice for $f$ on either of
$\bf {3}$ and  $\bf {6}$, then  there
are still at least two equivalence classes.
\end{proof}

 Another interesting example is furnished by the group
$ \Gamma = \langle a,b,c: a^2=b^2=c^2 \rangle$
 where  $\R$  is the ideal triangle with  vertices $0,1, \infty$
 and $a, b, c \in PSL(2,\ZZ)$ are elliptics of order two
 with  fixed points at $i$, $(1+i)/2$ and $1+i$ respectively.
In this case one checks  that
$f$ \emph{is} strictly irreducible.

\medskip

We base our general proof of strict  irreducibility on the following lemma.
\begin{lemma}
\label{lem:doublecover} Suppose that there exists a family of open intervals
$J_0, \ldots, J_m \in \I$ such that
\begin{enumerate}
\renewcommand{\labelenumi}{(\roman{enumi})}
\item $\cup_{i=0}^m f(J_i)$ covers $\dd \DD - \P$;
\item $f(J_i) \cap f(J_{i+1}) \neq \emptyset$ for $i=0, \ldots, m-1$.
\end{enumerate}
 Then
the Markov chain associated to the
 Markov map $f$   is   strictly irreducible.
\end{lemma}
\begin{proof}
By assumption (i),  every interval $I \in \I$ is equivalent to at least
one of the $J_i$.
By assumption (ii),  $J_i  \sim J_{i+1} $ for all $0 \le i < m$.
The result follows.
\end{proof}

\begin{prop}
\label{thm:strictlyirreducible}
Suppose that $|\dd \R | \ge 5$. Then
the Markov chain associated to the
 Markov map $f$   is   strictly irreducible.
\end{prop}
\begin{proof}
We show the sets $A(e), e \in \Gamma_0$, satisfy the requirements
of Lemma~\ref{lem:doublecover}.
Suppose that the extensions of the sides of $\R$ adjacent to
$e^{-1}$ meet $\dd \DD$ in points $V$ and $W$. Then  $f(A(e))$ is the interval between
$V$ and $W$ and not containing $L(e^{-1})$.
Since $|\dd \R | \ge 5$, this set of intervals overlaps round
 $\dd \DD$ in the required manner.
\end{proof}

 If $|\dd \R | = 4$, then $f$   may or may not be
   strictly irreducible. For example, suppose that  $\R$ has two opposite
    vertices $v,w \in \dd \DD$
 while  the other   opposite pair are in $\DD$.
 Suppose the sides adjacent to $v$ are
   paired, and equally the sides adjacent to $w$. Then one can verify
   directly that
   $\sim$ has two equivalence
   classes. The idea is that the points $v$ and $w$ divide $\dd \DD$ into two halves
   $E$ and $F$ say. One checks easily that
   the image of every interval  $I$ is contained either completely in $E$,
   or completely in $F$, so the intervals whose images fall in these
    two halves cannot  be equivalent.

   On the other hand, if $\R$ has $4$ sides all of which lie in $ \DD$,
   then by hypothesis we  assume that at least three geodesics
   in $\T$ meet at each vertex.
   One can check that the  images of the level one intervals at each vertex
   cover $\dd \DD$ in the manner required by Lemma~\ref{lem:doublecover}.

    We have already studied similar phenomena when $|\dd \R | = 3$.

\subsection{The alphabet map}
\label{sec:alphabet}

Finally we prove Proposition~\ref{thm:almostinjective}. We begin by recalling some further terminology from~\cite{BiS, BoS, STrieste}.

Let $e_{i_0}\ldots e_{i_n}$ be a word in the generators $\Gamma_0$.
Since $\Gamma_0$ consists of side pairing transformations of $\R$, the
regions $\R$ and $e_{i_r} \R$, and more generally
$e_{i_0}\ldots e_{i_{r-1}}\R$ and $e_{i_0}\ldots e_{i_{r}} \R$ for
$0 < r < n$,
 have a common side. The word
  $e_{i_0}\ldots e_{i_n}$
is called a \emph{cycle} if in the tesselation of $\DD$
by images of $\R$, the regions
$\R , e_{i_0} \R, e_{i_0} e_{i_1} \R ,\ldots, e_{i_0}\ldots e_{i_n} \R$
are arranged in order round a common vertex $v \in \DD$, see Section~\ref{sec:markov}.
(According to this definition, a single letter $e$ is always a cycle
provided that at least one of the vertices $\dd e$ is in $\DD$.)
Cycles $e_{i_0}\ldots e_{i_s}$, $e_{j_0 }\ldots e_{j_t}$ are called
\emph{consecutive} if there exists $ e \in \Gamma_0$ such that
$e_{i_0}\ldots e_{i_s}e$ and  $e^{-1} e_{j_0 }\ldots e_{j_t}$
are both cycles, see~\cite{BiS} for more details. This means that
$e_{i_0}\ldots e_{i_s}$ is a cycle at $v$ and that
$e_{j_0 }\ldots e_{j_t}$ is a cycle with the same orientation at $w$, where
$v$ and $w$ are the endpoints of the side $e^{-1}$ of $\R$, see Figure~\ref{fig:consecutive}.

\begin{figure}[hbt]
\centering 
\includegraphics[height= 6cm]{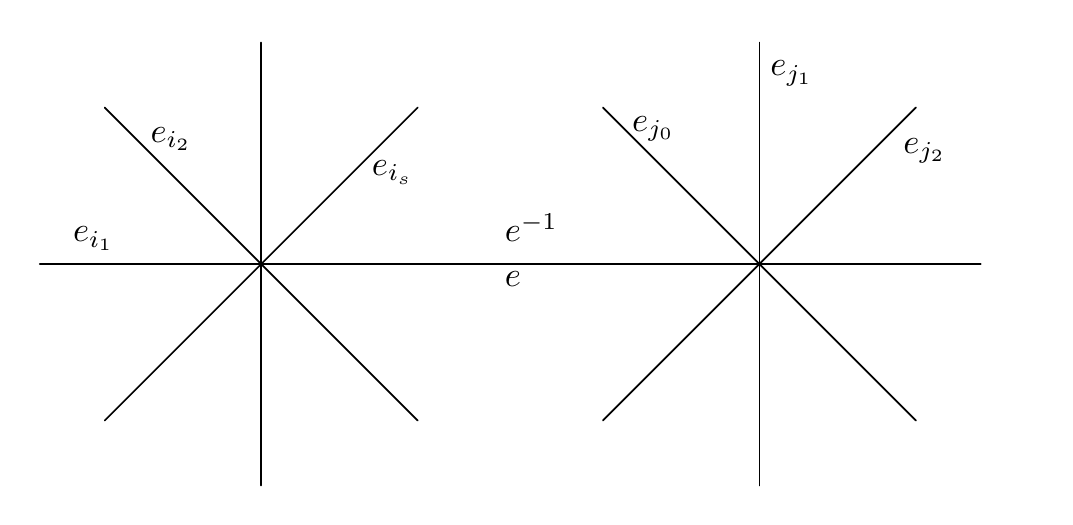} 
\caption{Consecutive cycles.}
 \label{fig:consecutive}
\end{figure} 

The word  $e_{i_0}\ldots e_{i_N}$
is called a \emph{special chain} if it consists of
a sequence of consecutive cycles
$B_1 B_2 \ldots B_n$  at vertices $v_1, \ldots, v_n$
such that $B_1$ has length at most  $n(v_1)-1$, $B_n$ has length at most  $n(v_n)$
and $B_i$ has length exactly  $n(v_i) -1$ for $1<i<n$.
 The geometrical meaning of this definition is that
 the sequences of copies of $\R$ corresponding
to a special chain
all touch a common hyperbolic line $t \subset \T$, all except possibly the first or last one
lying on the same side of $t$.

\begin{remark} {\rm  Special chains are intimately connected to the solution of the word
problem for Fuchsian groups given in~\cite{SInf, SMarkov, BiS}.
Special chains are shortest words and
two shortest words $V=e_{i_0}\ldots e_{i_N}$ and $W= e_{j_0}\ldots e_{j_N}$
with $e_{i_r} \neq  e_{j_r}$ for all $r$ represent the same element of $\Gamma$
only if either both $V$ and $W$ are single
cycles  of length $n(v)$, or are both  sequences of consecutive cycles
of lengths $n(v_1) -1,n(v_2) -1, \ldots,n(v_{n-1}) -1, n(v_n) $
and  $n(v_1),n(v_2) -1, \ldots,n(v_{n-1}) -1, n(v_n)-1 $
respectively. In the latter  case, the sequences of copies of $\R$ corresponding
to the words $ e_{i_0}\ldots e_{i_{N}}$ and
$e_{j_0}\ldots e_{j_N}$ meet along a common line in $\T$.}
\end{remark}

\begin{prop}
\label{prop:almostinjective}
Let $\pi: \Sigma \to \Gamma$ be the alphabet map. Then
$\pi^{-1}(e_{i_0}\ldots e_{i_N}) \le 2$ with equality if and only if
$e_{i_0}\ldots e_{i_N}$ ends in  a special chain.
\end{prop}

\begin{proof} [Proof of Proposition~\ref{thm:almostinjective}]
This is an immediate corollary   of Proposition~\ref{prop:almostinjective}.
 Notice that a special chain is completely specified
by its initial two letters (which determine
 the direction of the cycle at the first vertex) and
 the length of the initial cycle.
It follows that there exits a constant $K$ depending only on $\R$
such that
the total number of special chains of length exactly
$n$ is bounded by $K$,
independent of $n$. Since the total number of words in $\Gamma$
of length $n$ grows exponentially
with $n$, and since a special chain can be continued to
arbitrary length,  the result follows.
\end{proof}

We
establish several lemmas before proving Proposition~\ref{prop:almostinjective}.

 \begin{lemma}
 \label{lem:nevernever} Suppose that
$\pi (I_{i_0}\ldots I_{i_n})= \pi (I_{j_0}\ldots I_{j_m})$.
Then $n=m$ and
$\pi (I_{i_r})= \pi (I_{j_r})$ for $r=0, \ldots,n$. Moreover if
$I_{i_k} \neq I_{j_k}$ for some $k$ then $I_{i_r} \neq I_{j_r}$ for any $r>k$.
 \end{lemma}
 \begin{proof}
  Suppose that
$\pi (I_{i_0}\ldots I_{i_n})= \pi (I_{j_0}\ldots I_{j_m})$. Since the
images of both sequences are shortest, $n=m$.
Moreover because of unique representation in $\Gamma$ by sequences in the image of $\pi$, see Theorem~\ref{thm:hypothesis}, we have
$\pi (I_{i_r})= \pi (I_{j_r})$ for $r=0, \ldots,n$.

Now suppose that $I_{i_k} \neq I_{j_k}$. By definition,  $f|_I = \pi(I)^{-1}$.
 Since   $\pi (I_{i_k})= \pi (I_{j_k})$,
we see  that $ f$ is injective on  $I_{i_k} \cup I_{j_k}$ and the result
follows.
\end{proof}

\begin{lemma}
\label{lem:3steps}
Suppose that
$\pi (I_{i_0}\ldots I_{i_n})= \pi (I_{j_0}\ldots I_{j_n})$ with
 $ I_{i_0}  \neq  I_{j_0}$, and suppose
 that $0 \le r <n$. Then: \begin{enumerate}
\renewcommand{\labelenumi}{(\roman{enumi})}
\item if $I_{i_r} \subset C(v) \cap L(e) $ for $e \in \Gamma_0$
and   $v \in \dd e$, then
 either $I_{j_r} \subset C(v)$  or $I_{j_r} = A(e)$;
\item if $I_{i_r} = A(e) $ for $e \in \Gamma_0$
then $I_{j_r}$  is adjacent to $A(e) $;
\item if  $I_{i_r}, I_{j_r} \subset C(v) \cap L(e) $ for $e \in \Gamma_0$
and some vertex $v \in \dd e$, and  if
$\lev(I_{j_r} ) < lev (I_{i_r}) = k$ and $r+k \le n$,
then $f^{k-1}(I_{i_r}) =I_{i_{r+k-1}} = A(d)$ for some $d \in \Gamma_0$.
\end{enumerate}
\end{lemma}
\begin{proof}

\noindent \emph{Assertion (i):} Let $w$ be the other vertex in $\dd e$. If the result is false, then
$I_{j_r} \subset C(w) \cap L(e) $. Let $s(x)$ and $s(y)$ be the sides of $\R$
adjacent to $s(e^{-1})$.
Note that $f|_{I_{j_r}} = e^{-1}$ so that
 $f(C(v) \cap L(e))  \subset L(x)$, say, and $f(C(w) \cap L(e) ) \subset L(y)$.

By the hypotheses  of Theorem~\ref{thm:hypothesis}
either $ |\dd \R | > 3$ or $\R$ has a vertex at $\infty$.
In both cases
  $L(x)$ and $L(y)$ are disjoint. In the first case this is clear by Lemma~\ref{lem:nomeeting}. For the second,
   observe that  we may assume that both vertices in $\dd e^{-1}$ are in $ \DD$,
   since otherwise at least one  of $C(v)$ and  $C(w)$
   is empty and there is nothing to prove. This forces $\pi(I_{i_{r+1}} )\neq \pi(I_{j_{r+1}} )$ contrary to  hypothesis, which gives the result.

\noindent \emph{Assertion (ii):} Suppose first that $e \neq e^{-1}$. Observe that the image under $e^{-1}$ of any interval in $I \subset L(e) $ but
not adjacent to   $A(e) $ is contained in   $C(\dd e^{-1})$
and is thus contained in $L(e^{-1}) \cup L(x) \cup L(y)$ where as above $s(x),s(y)$
are the two sides of $\R$ adjacent to
$s(e^{-1})$.
On the other hand, the image  under $e^{-1}$ of $A(e)$ is outside
$L(e^{-1}) \cup L(x) \cup L(y)$. Thus provided $L(x)$ and $ L(y)$ are disjoint, $\pi(J)  \notin \{ e,x,y\}$ for any $J \in \I$ contained in $ f(A(e))$.
The result follows as above if $|\dd \R| >3$. The special case $|\dd \R|=3$
is easily treated separately.

Finally suppose  that $e= e^{-1}$. If $I \subset C(\dd e)$ then $f(I) \subset C(\dd e)$
but $f(A(e))$ is \emph{outside} $L(e)$ which is impossible. 

\noindent \emph{Assertion (iii):}  The map $f$ decreases level
and at each stage with $ t < k$, $I_{i_{r+t}}$ and $I_{j_{r+t}}$ are in the  crown
of a common vertex.
At  step $k-1$,  $I_{i_{r+k-1}}$ has level $1$
so that by (i), $I_{i_{r+k-1}} = A(d)$ for some $d \in \Gamma_0$.
\end{proof}

\begin{proof} [Proof of Proposition~\ref{prop:almostinjective}]
Suppose that
$$\pi (I_{i_0}\ldots I_{i_n})= \pi (I_{j_0}\ldots I_{j_n}).$$
 Without loss of generality
we may as well assume
that $I_{i_0} \neq I_{j_0}$.
By Lemma~\ref{lem:nevernever}, $I_{i_r} \neq I_{j_r}$ for any $r>0$.

Suppose first that $r<n$ and that $I_{i_r} \subset C(v) \cap L(e)$
has level $k >1$ in the clockwise direction starting at the highest level interval
at vertex $v$. (The proof if the interval is in the anticlockwise direction is similar; obviously the direction of all cycles are then reversed.) Then $I_{i_{r+1}}$ has level $k -1$  at vertex
$e^{-1} v$ and is outside $L(e^{-1})$. Thus $f(I_{i_r}) \subset L(e')$
where $s(e')$ is the  side of $\R$ adjacent to $s(e^{-1})$ in clockwise order
round $\dd \R$.
Therefore $e^{-1} e'^{-1} = \pi (I_{i_r})\pi ( I_{i_{r+1}})$
is an anticlockwise cycle at  $e^{-1} v$.

Inductively, it follows that $\pi (I_{i_r}) \pi ( I_{i_{r+1}})\ldots \pi ( I_{i_{r+k-1}})$
is an anticlockwise cycle. Further, by Lemma~\ref{lem:3steps}  (iii)
we see that $I_{i_{r+k-1}} = A(d)$  for some $d \in \Gamma_0$ and
that $I_{j_{r+k-1}}$ is adjacent to $I_{i_{r+k-1}}$, where
$d^{-1}= \pi ( I_{i_{r+k-1}})$.

 Let the sides of $\dd \R$
in clockwise order
from $s(d)$ be  $s(d),s(c),s(b)$ (so that the interior labels of sides in clockwise order are 
$d^{-1}, c^{-1},b^{-1}$).
Let $v$ be the common vertex of  $s(d)$ and $s(c)$
and let $w$ be the common vertex of  $s(c)$ and $s(b)$.
From the proof of Lemma~\ref{lem:3steps} (ii), it follows that
$I_{j_{r+k}}$ is  necessarily the highest level interval at $v$,
and that $\pi(I_{j_{r+k}}) = c$.

Now  $\pi ( I_{i_{r+k-1}})\pi(I_{j_{r+k}}) = d^{-1} c$ is an anticlockwise cycle  at $v$, hence so also is \\
 $\pi (I_{i_r})\pi ( I_{i_{r+1}})\ldots \pi ( I_{i_{r+k-1}})c$. Furthermore $ c^{-1}\pi(I_{j_{r+k}}) =  c^{-1}b$
is an anticlockwise cycle at $w$. This means the cycles
$$\pi (I_{i_r})\pi ( I_{i_{r+1}})\ldots \pi ( I_{i_{r+k-1}}) \ \ {\rm and}
\ \ \pi(I_{j_{r+k}}) = \pi(J_{j_{r+k}})$$ are consecutive. Moreover
 $I_{i_{r+k}}$ and $I_{j_{r+k}}$ are both contained in the crown
 at $w$, in fact
$I_{i_{r+k}}$ is the    level $n(w) -1$ interval in the crown at $w$,
adjacent to $I_{j_{r+k}}$ going around clockwise.

Now the first part of the argument repeats so that
$\pi(I_{j_{r+k}})$ is (assuming $n$ is sufficiently large)
the first term in  an anticlockwise cycle of length
of length $n(w) -1$.
A similar argument   in the case
$I_{i_0} = A(e)$ completes the proof.
\end{proof}

\subsection{An Ergodic Theorem for Fuchsian Groups}
\label{sec:final}
As before, let $\Gamma$  be a finitely generated non-elementary 
Fuchsian group acting in the hyperbolic disk $\DD$, and assume
that a fundamental domain  $\R$ for $\Gamma$ has even corners and satisfies $|\dd \R| \ge 5$.
As before, let $\Gamma_0$ be the generating set corresponding to $\R$.  For $g\in \Gamma$,
let $|g|$ be the length of the shortest word in $\Gamma_0$ representing $\Gamma$, and
for $n\in {\mathbb N}$, let
$$
S(n)=\{g\in\Gamma: |g|=n\}
$$
be the sphere of radius $n$ in $\Gamma$. Finally,  let
$K_n$ be the cardinality of $S(n)$. Observe that $K_n$ grows exponentially
and so, by the Borel-Cantelli Lemma, the contribution of ``non-injective" elements of Proposition~\ref{thm:almostinjective} to the spherical averages is negligible.
Propositions \ref{matfun},
 \ref{thm:almostinjective} and~\ref{thm:strictlyirreducible} now
imply
\begin{introthm}
\label{generalfuchsian}
Let $\Gamma$ act ergodically on a probability space $(X, \nu)$ by measure-preserving
transformations, and, for $g\in\Gamma$, let $T_g$ be the corresponding transformation.
Suppose $\Gamma_0$ is a geometric set of generators associated to a Markov partition as in Section~\ref{sec:markov}, and suppose either that $|\dd \R| \geq 5$, or, if $|\dd \R| < 5$, that  the associated transition matrix is strictly irreducible. Then, measuring word length with respect to the generators $\Gamma_0$,  for any $\varphi\in L_1(X, \nu)$ we have
\begin{equation}
\frac1N\sum_{n=0}^{N-1}\frac1{K_n} \sum\limits_{g\in S(n)} \varphi\circ T_g   \to \int\limits_X fd\nu
\end{equation}
both $\nu$-almost surely and in $L_1(X, \nu)$ as $N\to\infty$.
\end{introthm}

Theorem \ref{generalfuchsian} implies as a special case Theorem \ref{main}.

\section*{Acknowledgements}
A.I. B. is an Alfred P. Sloan Research Fellow.
During work on this project, he
was supported in part by Grant MK-4893.2010.1 of the President of the Russian Federation,
by the Programme on Mathematical Control Theory of the Presidium of the Russian Academy of Sciences,
by the Programme 2.1.1/5328 of the Russian Ministry of Education and Research,  by the Edgar Odell Lovett Fund at Rice University, by
the National Science Foundation under grant DMS~0604386, and by 
the RFBR-CNRS grant 10-01-93115.


\begin{thebibliography}{}

\bibitem{Bea} A.~Beardon.
\newblock	 An introduction to hyperbolic geometry.
\newblock  In Ergodic
Theory and Symbolic Dynamics in Hyperbolic Spaces,  T. Bedford, M. Keane and C. Series  eds., Oxford Univ.
 Press, 1991.


\bibitem{BiS} J.~Birman and C.~Series.
\newblock Dehn's algorithm revisited, with application
to simple curves on surfaces.
\newblock Combinatorial Group Theory and
Topology, S. Gersten and J. Stallings eds., Ann. of Math. Studies III,
Princeton U.P., 1987, 451--478.

\bibitem{Bow} L.~Bowen.
Invariant measures on the space of horofunctions of a word hyperbolic group.
Ergodic Theory Dynam. Systems 30 (2010), no. 1, 97--129.


\bibitem{BoS}
 R.~Bowen and C.~Series.
\newblock Markov maps associated with
Fuchsian groups.
\newblock  IHES Publications, 50, 1979, 153--170.

\bibitem{bufann}
A.I.~Bufetov.
Convergence of spherical averages for actions of free groups. Ann. of Math. (2), 155, 2002, 929--944.

\bibitem {bufrokh} A.I.~Bufetov. Markov averaging and ergodic theorems for several operators. Topology, ergodic theory, real algebraic geometry, 39--50, Amer. Math. Soc. Transl. Ser. 2, 202, Amer. Math. Soc., Providence, RI, 2001.

\bibitem {fune} K.~Fujiwara and A.~Nevo. Maximal and pointwise ergodic theorems for word-hyperbolic groups. Ergodic Theory and Dynamical Systems,   18, 1998,  843--858.


\bibitem {gor-nev} A.~Gorodnik and A.~Nevo.
The ergodic theory of lattice subgroups.
Annals of Mathematics Studies, 172. Princeton University Press, Princeton, NJ, 2010.


\bibitem{grigor}R. I.~Grigorchuk. Ergodic theorems for the actions of a free group and a free semigroup. (Russian) Mat. Zametki 65, 1999, no. 5, 779--783; translation in Math. Notes 65, 1999, 654--657.

\bibitem{grigor1} R. I.~Grigorchuk. An ergodic theorem for actions of a free semigroup. (Russian) Tr. Mat. Inst. Steklova 231 (2000), Din. Sist., Avtom. i Beskon. Gruppy, 119--133; translation in Proc. Steklov Inst. Math. 2000, no. 4 (231), 113--127.

\bibitem{mns} G. A.~Margulis, A.~Nevo and E.M.~Stein. Analogs of Wiener's ergodic theorems for semisimple Lie groups. II. Duke Math. J. 103 (2000), no. 2, 233--259. 


\bibitem{nevo1} A.~Nevo. Harmonic analysis and pointwise ergodic theorems for noncommuting transformations. J. Amer. Math. Soc. 7 (1994), no. 4, 875--902.


\bibitem{nevo}
A.~Nevo. Pointwise ergodic theorems for actions of groups. Handbook of dynamical systems. Vol. 1B, 871--982, Elsevier B. V., Amsterdam, 2006.

 
\bibitem{ns}
 A.~Nevo and E.M.~Stein. A generalization of Birkhoff's pointwise ergodic theorem. Acta Math. 173 (1994), no. 1, 135--154.
 
 \bibitem{ns1}
 A.~Nevo and E.M.~Stein. Analogs of Wiener's ergodic theorems for semisimple groups. I. Ann. of Math. (2) 145 (1997), no. 3, 565--595.



\bibitem{SInf} C.~Series.
\newblock The infinite word problem and limit sets in
Fuchsian groups.
\newblock Ergodic Theory and Dynamical Systems, 1, 1981, 337--360.

\bibitem{SMartin} C.~Series.
\newblock	Martin boundaries of random walks on Fuchsian groups.
\newblock Israel J. Math. 44, 1983, 221--240.

\bibitem{SMarkov} C.~Series.
\newblock	Geometrical Markov coding of geodesics on surfaces of
 constant negative curvature.
 \newblock Ergodic  Theory and Dynamical Systems, 6,
 1986, 601--625.

\bibitem{STrieste} C.~Series.
\newblock	 Geometrical methods of symbolic coding.
\newblock  In Ergodic
Theory and Symbolic Dynamics in Hyperbolic Spaces,   T. Bedford, M. Keane and C. Series eds., Oxford Univ.
 Press, 1991.




\end{thebibliography}
\end{document}